\documentclass[a4paper,12pt]{article}
\usepackage{times, url}
\usepackage{amssymb,latexsym,amsmath,epsfig,amsthm,mathrsfs}
\usepackage{amsfonts}
\usepackage{amscd}
\usepackage{dsfont}
\usepackage{bbm}
\usepackage{rotating}
\usepackage{newlfont}
\usepackage{enumitem}

\newtheorem{thma}{Theorem}

\newtheorem{thm}{Theorem}[section]

\newtheorem{lem}{Lemma}[section]

\def\sm{\setminus}

\usepackage{lineno,hyperref}
\modulolinenumbers[5]

\begin{document}
\setcounter{page}{1}

\begin{center}
{\LARGE \bf  Restricted-sum-dominant sets}
\vspace{8mm}

{\large \bf Raj Kumar Mistri$^1$ and R. Thangadurai$^2$}
\vspace{3mm}

$^1$ Department of Mathematics, Harish-Chandra Research Institute, HBNI, \\
Chhatnag Road, Jhunsi, Allahabad - 211 019, India \\
e-mail: \url{itsrajhans@gmail.com}
\vspace{2mm}

$^2$ Department of Mathematics, Harish-Chandra Research Institute,HBNI, \\
Chhatnag Road, Jhunsi, Allahabad - 211 019, India \\
e-mail: \url{thanga@hri.res.in}
\vspace{2mm}


\end{center}
\vspace{10mm}

\noindent
{\bf Abstract:} Let $A$ be a nonempty finite subset of an additive abelian group $G$. Define $A + A := \{a + b : a, b \in A\}$ and $A \dotplus A := \{a + b : a, b \in A~\text{and}~ a \neq b\}$. The set $A$ is called a {\em sum-dominant (SD) set} if $|A + A| > |A - A|$, and it is called a {\em restricted sum-domonant (RSD) set} if $|A \dotplus A| > |A - A|$. In this paper, we prove that for infinitely many positive integers $k$, there are  infinitely many RSD sets of integers of cardinality $k$. We also provide an explicit construction of infinite sequence of RSD sets.
\vspace{10mm}

\section{Introduction}

A set $A$ is called a {\em More Sums Than Differences (MSTD) set} or a {\em sum-dominant (SD) set} if $|A + A| > |A - A|$, and it is called a {\em restricted sum-domonant (RSD) set} if $|A \dotplus A| > |A - A|$.

In the following discussion, we will restrict ourselves to the additive group of integers $G = \mathbb{Z}$, if not specified. As $a + b = b + a$, but $a - b \neq b - a$ unless $a = b$, it is natural to guess that $|A + A| > |A - A|$ and $|A \dotplus A| > |A - A|$. But neither of these assertions is true in general.

For examples, if $A = \{0, 2, 3\}$, then $|A \dotplus A|< |A + A| < |A - A|$. But if $A' = \{0, 2, 3, 4, 7, 11, 12, 14\}$, then $|A' + A'| > |A' - A'| > |A' \dotplus A'|$. Again, if
\[A'' = \{0, 1, 2, 4, 5, 9, 12, 13, 17, 20, 21, 22, 24, 25, 29, 32, 33, 37, 40, 41, 42, 44, 45\},\]
then $|A'' + A''| > |A'' \dotplus A''| > |A'' - A''|$. Similarly, if $P$ is a set with $k \geq 2$ elements in arithmetic progression, then it is known that (see \cite{nath})
\[|P \dotplus P| = 2k - 3~\text{and}~ |P - P| = 2k -1.\]
Hence it is not an RSD set.

For an integer $c$ and a set $A$, let $c + A:= \{c + a: a \in A\}$ and $c\cdot A := \{ca : a \in A\}$. Two sets $A$ and $B$ are called {\em affinely equivalent sets} if there exist integers $x$ and $y \neq 0$ such that either $A = x + r \cdot B$ or $B = x + r \cdot A$. If $A$ and $B$ are not affinely equivalent, then they are called {\em affinely inequivalent sets}. Let $A$ and $B$ be affinely equivalent sets. Then $A$ is an SD set (respectively, RSD set) if and only if $B$ is an SD set (respectively, RSD set).

A large number of results exist regarding the existence and infinitude of SD sets. See (\cite{hegarty2007some,hegarty2009when,Marica1969on, martin2007many, miller2010explicit, miller2010explicitconstr, nathanson2007problems, nathanson2007sets, nathanson2017problems, zhao2010constr,zhao2010counting, zhao2011sets}), and references given therein. The restricted-sum-dominant sets of integers have been studied recently by D. Penman and M. Wells\cite{penman2013onsetswith}. In case of finite abelian groups, some results are available in the literature due to Nathanson \cite{nathanson2007sets}, Hegarty\cite{hegarty2007some}, Zhao\cite{zhao2010counting}, Penman and Wells\cite{penman2014sumdominant}. The term ``MSTD Sets" was introduced by Nathanson\cite{nathanson2007problems}. John Marica (1969) wrote the first paper \cite{Marica1969on} on MSTD Sets. In this paper, we will use the term ``SD set" instead of ``MSTD set."

Because of the commutativity of addition and noncommutativity of subtraction, it was natural to conjecture that SD sets are rare. But this is false as shown by Martin and O'Bryant\cite{martin2007many}.

As already mentioned there are several constructions of SD sets, but the complete classification of these sets are still unknown. In this direction, Hegarty\cite{hegarty2007some} proved the following result.
\begin{thma}
There is no SD set of integers of cardinality $7$. Moreover, any SD set of cardinality $8$ is affinely equivalent to the set $A = \{0, 2, 3, 4, 7, 11, 12, 14\}$.
\end{thma}

Now suppose that there is an RSD set of cardinality $7$, then that set will be also an SD set of cardinality $7$ which contradicts the above result. Moreover, it is easy to verify that the SD set $A = \{0, 2, 3, 4, 7, 11, 12, 14\}$ is not an RSD set. Hence there is no RSD set of cardinality $8$ also. D. Penman and M. Wells\cite{penman2013onsetswith} proved that there is no RSD set of cardinality $9$ also.
\begin{thma} \label{nonexistence-rsd-sets-thm}
  There is no RSD set of integers of cardinality $7, 8$ and $9$.
\end{thma}

For positive integers $k$ and $n$, let $H(k, n)$ denote the number of affinely inequivalent SD sets of integers of cardinality $k$ contained in the set $\{0, 1, \ldots, n\}$, and let Let $H(k)$ denote the number of affinely inequivalent SD sets of integers of cardinality $k$. Similarly, For posititive integers $k$ and $n$, let $H^*(k, n)$ denote the number of affinely inequivalent RSD sets of integers of cardinality $k$ contained in the set $\{0, 1, \ldots, n\}$, and let $H^*(k)$ denote the number of affinely inequivalent RSD sets of integers of cardinality $k$.

Nathanson \cite{nathanson2017problems} proved that if there exists an SD set of integers of cardinality $k$, then there exist infinitely many SD sets of integers of cardinality $k^n + 1$ for all integers $n \geq k$, that is, $H(k^n + 1) = \infty$ for all $n \geq k$. We prove the similar result for RSD sets stated below.

\begin{thm}\label{infinetly-many-rsd-sets-thm}
  If there exists an RSD set of integers of cardinality $k$, then there exist infinitely many RSD sets of integers of cardinality $k^n + 1$ for all integers $n \geq k$, that is $H^*(k^n + 1) = \infty$ for all $n \geq k$.
\end{thm}
We prove this result in the next section. The proof is similar to that of Nathanson's result.

Martin and O'Bryant\cite{martin2007many} showed that the range of possible values for $|A + A| - |A - A|$ is all of $\mathbb{Z}$. D. Penman and M. Wells\cite{penman2013onsetswith} proved by construction the same result for the quantity $|A \dotplus A| - |A - A|$. In this paper, we prove this result by providing a different construction of sets which is a slight modification of the construction provided by Martin and O'Bryant\cite{martin2007many}(see Theorem \ref{constucted-set}). The result in turn proves the infinitude of RSD sets.

We have already seen some examples of RSD sets. In the next section, we prove a conditional existence result for RSD sets with certain cardinalities.

\section{Existence of infinitely many RSD sets of certain cardinalities} \label{existence-rsd-sets}
Let $\Delta(A) = |A - A| - |A \dotplus A|$. The set $A$ is an RSD set if and only if $\Delta(A) < 0$.

Let $A = \{a_0, a_1, \ldots, a_{k-1}\}$ and $A' = \{a_0, a_1, \ldots, a_{k-1}, a_k\}$ be two sets such that $a_k > a_{k-2}+ a_{k-1}$, where $k \geq 3$ and $0 = a_0 < a_1 < a_2 < \cdots < a_{k-1} < a_k$.
Now
\[A'\dotplus A' = (A \dotplus A) \cup \{a_k + a_i : i = 0, 1, \ldots, k-1\},\]
and
\[\max(A \dotplus A) = a_{k-2}+ a_{k-1} < a_k < a_k + a_1 < \cdots a_k + a_{k-1}.\]
Hence $|A' \dotplus A'| = |A \dotplus A| + k$.
Now
\[A'- A' = (A - A) \cup \{\pm(a_k - a_i) : i = 0, 1, \ldots, k-1\},\]
and so
\[|A' - A'| = |A - A| + 2k.\]
Therefore,
\[\Delta(A') = |A' - A'| - |A' \dotplus A'| = |A - A| + 2k - |A \dotplus A| - k = \delta(A) + k.\]
We have proved the following
\begin{lem}
 Let $A = \{a_0, a_1, \ldots, a_{k-1}\}$ be a set of $k$ integers, where $k \geq 3$ and $0 = a_0 < a_1 < a_2 < \cdots < a_{k-1}$. If $a_k$ is an integer such that $a_k > a_{k-2}+ a_{k-1}$, and if $A' = A \cup {a_k}$, then
 \[\Delta(A') = \Delta(A) + k.\]
\end{lem}

Since $\Delta(A') = \Delta(A) + k$, it is easy to see that $\Delta(A') < 0$ if and only if $|A \dotplus A| \geq |A - A| + k + 1$.
Now the follows lemma follows easily
\begin{lem}\label{existence-lemma}
  If there exists an RSD set $A$ of integers with $|A \dotplus A| \geq |A - A| + |A| + 1$, then there exist infinitely many RSD sets of integers of cardinality $|A| + 1$.
\end{lem}

\begin{lem}
  Let $A$ be a nonempty finite set of nonnegative integers with $a^* = \max(A)$. Let $m$ be a positive integer with $m > 2a^*$. If $n$ is a positive integer and
  \begin{equation}\label{constucted-set}
    B := \left\{\sum_{i=0}^{n-1}a_i m^i : a_i \in A ~\text{for all}~ i = 0, 1, \ldots, n-1\right\},
\end{equation}
  then
  \begin{align*}
    |B| &= |A|^n, \\
    |B \dotplus B|&\geq |A \dotplus A|^n, \\
    |B - B| &= |A - A|^n.
  \end{align*}
\end{lem}
\begin{proof}
  The first identity follows easily from the uniqueness of the $m$-adic representation of an integer. For the second identity, note that any element $b \in B \dotplus B$ is of the form $b_0 + b_1 m + b_2 m^2 + \cdots + b_{n-1} m^{n-1}$, where $b_i \in A + A$ for all $i = 0, 1, \ldots, n-1$ and $b_i \in A \dotplus A$ for at least one $i$. Therefore, $|B \dotplus B|\geq |A \dotplus A|^n$. For the proof of the third identity one can see Lemma $4$ in \cite{nathanson2016problems}
\end{proof}

\begin{proof}[Proof of Theorem \ref{infinetly-many-rsd-sets-thm}]
  Let $A$ be set of integers with $|A| = k \geq 2$. Without loss of generality, we may assume that $\min(A) = 0, \gcd(A) = 1$ and $\max(A) = a^*$. Let $m$ and $n$ be positive integers such that $m > 2a^*$ and $n \geq k$, respectively. Let $B$ be the set defined by $(\ref{constucted-set})$. Then by Lemma \ref{constucted-set}, we have
  \begin{align*}
    |B \dotplus B| &\geq |A \dotplus A|^n \\
     &\geq (|A - A| + 1)^n \\
     & \geq |A - A|^n + n |A - A|^{n-1} + 1 \\
     & \geq |A - A|^n + n(2k - 1)^{n-1} + 1 \\
    & \geq |A - A|^n + k^n + 1\\
    &= |B - B| + |B| + 1.
  \end{align*}
Therefore, $B$ is an RSD set and hence by Lemma \ref{existence-lemma}, there exist infinitely many RSD sets of cardinality $k^n + 1$. Since we have infinitely many choices for $m$ and $n$, it follows that there are infinitely many RSD sets of cardinality $k^n + 1$ for all $n \geq k$. This completes the proof.
\end{proof}

\section{Range of $|A \dotplus A| - |A - A|$} \label{infinitude-rsd-sets}
The next natural question is whether there are infinitely many RSD sets. It is known that there are infinitely many SD sets. We shall show that for every integer $x$, there exists a set $A_x$ such that $\Delta(A_x) = -x$ which in turn implies that there are infinitely many RSD sets. This also shows that the range of $|A \dotplus A| - |A - A|$ is whole of $\mathbb{Z}$.  Note that for distinct integers $x$ and $y$, the corresponding sets $A_x$ and $A_y$ will be distinct, otherwise $-x = \Delta(A_x) = |A_x - A_x| - |A_x \dotplus A_x| = |A_y - A_y| - |A_y \dotplus A_y| = -y$, a contradiction. Thus for each positive integer $x$, there exists an RSD set $A_x$, and hence it follows that there are infinitely many RSD sets.

\begin{thm}
  For any integer $x$, there exists a nonempty finite set $A_x \subseteq \mathbb{Z}$ such that $\Delta(A_x) = -x$.
\end{thm}
\begin{proof}
  For $x = -1$, we can take $A_{-1} = \{0\}$. Now Let $x \leq -2$. Consider the set $A_x = \{0, 1, \ldots, -x -1\} \cup \{-2x -2\}$. Then
  \[A_x \dotplus A_x = [1, 3x -3]~\text{and}~ A_x - A_x = [2x - 2, -2x + 2],\]
  and so
  \[\Delta(A_x) = |A_x - A_x|- |A_x \dotplus A_x| = 3x - 3 - (4x - 3) = -x.\]

\textbf{Odd positive values of $x$:}
Let $x = 2k + 1$. Define
\begin{align*}
  A_{2k + 1} &= \{0, 2, 3, 4, 7, 11, 12, 14\} + \{0, 29, 58, \ldots, 29(k+3)\}\\
 &= \{0 \leq s \leq 29(k+3) + 14: s \equiv 0, 2, 3, 4, 7, 11, 12 ~\text{or}~ 14 \pmod{29}\}.
\end{align*}
Then
\begin{align*}
A_{2k + 1} \dotplus A_{2k + 1} &= (A_{2k + 1} + A_{2k + 1}) \sm \{0, 8, 20, 22, 24, 28\}\\
&= \{0 \leq s \leq 29(2k + 7): s \not \equiv 1, 20 ~\text{or} 27 \pmod{29}\} \sm \{0, 8, 20, 22, 24, 28\},
\end{align*}
and
\[A_{2k + 1} - A_{2k + 1} = \{-29(k + \frac{7}{2}) < s < 29(k + \frac{7}{2}): s \not \equiv -13, -6, 6 ~\text{or} 13 \pmod{29}\}.\]
Therefore,
\[\Delta(A_{2k + 1})= |A_{2k + 1} - A_{2k + 1}| - |A_{2k + 1} \dotplus A_{2k + 1}| =25(2k + 7) -26(2k + 7) + 6 = -(2k + 1).\]

\textbf{Even positive values of $x$:}
Let $x = 2k$ with $k \geq 4$. Define
\[A_{2k} = A_{2k + 1} \sm \{29\}.\]
Then
\[A_{2k} \dotplus A_{2k} = (A_{2k + 1} \dotplus A_{2k + 1}) \sm \{29\}.\]
Therefore,
\[|A_{2k} \dotplus A_{2k}|= 26(2k + 7) - 6 - 1 = 26 (2k + 7) - 7,\]
and
\[A_{2k}- A_{2k}= |A_{2k + 1} - A_{2k + 1}| = 25(2k + 7).\]
Hence
\[\Delta(A_{2k})= |A_{2k} - A_{2k}| - |A_{2k} \dotplus A_{2k}| =25(2k + 7) -26(2k + 7) + 7 = -2k.\]
\end{proof}

\section{Concluding remarks}
Nathanson poses some problems for SD sts in \cite{nathanson2017problems}. Motivated by these problems, we can formulate the following problems. By Theorem \ref{nonexistence-rsd-sets-thm}, we know that $H^*(k, n) = 0$ for $k = 7, 8, 9$. It would be an interesting problem to compute $H^*(k, n)$ for $k \geq 10$. Note also that
\[H^*(k) = \lim_{n \rightarrow \infty} H^*(k, n).\]
Thus $H^*(k) = \infty$ if there exist infinitely many affinely inequivalent RSD sets of integers of cardinality $k$. It would be also an interesting problem to study the behaviour of  $H^*(k)$. For example, one can ask whether there exist infinitely many affinely inequivalent RSD sets of integers of cardinality $k$ for all sufficiently large $k$. One can also try to determine the smallest $k$ for which $H^*(k) = \infty$.

\section*{Acknowledgements}

This research of the first named author was supported by the PDF Scheme (Letter No. HRI/4041/3761) of HRI.

\makeatletter
\renewcommand{\@biblabel}[1]{[#1]\hfill}
\makeatother


\begin{thebibliography}{99}


\bibitem{hegarty2007some} P. V. Hegarty, Some explicit constructions of sets with more sums than differences, {\it Acta Arith.\/} 130 (2007), 61-77.
%
\bibitem{hegarty2009when} P. Hegarty, S. J. Miller, When almost all sets are difference dominated, {\it Random Structures Algorithms\/} 35 (2009), no. 1, 118-136.
%
\bibitem{Marica1969on}J. Marica, On a conjecture of Conway, Canad. Math. Bull. 12 (1969), 233-234.
%
\bibitem{martin2007many} G. Martin and K. O'Bryant, Many sets have more sums than differences. In {\em Additive Combinatorics, CRM Proceedings {\em \&} Lecture Notes 43;} American Mathematical Society, Providence, RI, 2007; pp. 287-305.

%
\bibitem{miller2010explicit} S. J. Miller, B. Orosz, D. Scheinerman, Explicit constructions of infinite families of MSTD sets, {\it J. Number Theory\/} 130 (2010), no. 5, 1221-1233.
%
\bibitem{miller2010explicitconstr} S. J. Miller and D. Scheinerman, Explicit constructions of infinite families of MSTD sets. In {\em Additive number theory;} Springer, New York, 2010; pp. 229-248.
%
\bibitem {nath} M. B. Nathanson, Additive Number Theory: Inverse Problems and the Geometry of Sumsets, {Springer, 1996\/}.
%
\bibitem{nathanson2007problems} M. B. Nathanson, Problems in additive number theory. I. In {\em Additive Combinatorics, CRM Proceedings {\em \&} Lecture Notes 43;} American Mathematical Society, Providence, RI, 2007; pp. 263-270.
%
\bibitem{nathanson2007sets} M. B. Nathanson, Sets with more sums than differences, {\it Integers\/} 7 (2007), A5, 24 pp.
%
\bibitem{nathanson2017problems} M. B. Nathanson, Problems in Additive Number Theory. V: Affinely Inequivalent MSTD Sets, {\it North-West. Eur. J. Math.\/} 3 (2017), 123-141.
%
\bibitem{penman2013onsetswith}D. B. Penman and M.D. Wells, On sets with more restricted sums than differences, {\it  Integers\/} 13 (2013), A57, 24 pp.
%
\bibitem{penman2014sumdominant} D. B. Penman and M.D. Wells, Sum-dominant sets and restricted-sum-dominant sets in finite abelian groups, {\it  Acta Arith.\/} 165 (2014), no. 4, 361-383.
%
\bibitem{zhao2010constr} Y. Zhao, Constructing MSTD sets using bidirectional ballot sequences, {\it J. Number Theory\/} 130 (2010), 1212-1220.
%
\bibitem{zhao2010counting} Y. Zhao, Counting MSTD sets in finite abelian groups, {\it J. Number Theory\/} 130 (2010), 2308-2322.
%
\bibitem{zhao2011sets} Y. Zhao, Sets characterized by missing sums and differences, {\it J. Number Theory\/} 131 (2011), 2107-2134.

\end{thebibliography}
\end{document}